\documentclass[12pt,reqno]{amsart}

\usepackage[english]{babel}
\usepackage{mathrsfs}
\usepackage{indentfirst}
\usepackage{paralist}
\usepackage{amssymb}
\usepackage{amsthm}
\usepackage{amsmath}
\usepackage{amscd}
\usepackage{graphicx}
\usepackage[colorlinks=true]{hyperref}
\usepackage[margin=1in]{geometry}
\usepackage[utf8]{inputenc}
\usepackage{cite}
\usepackage{enumitem} 
\usepackage{multicol}
\usepackage{xcolor}
\definecolor{ForestGreen}{rgb}{0.13, 0.55, 0.13}

\theoremstyle{plain} 
\newtheorem{theorem}{Theorem}[section] 
\newtheorem*{theorem*}{Main Result}
\newtheorem{thm*}{Known result}
\newtheorem{corollary}[theorem]{Corollary} 
\newtheorem{lemma}[theorem]{Lemma}

\newtheorem{definition}[theorem]{Definition}

\theoremstyle{definition}

\theoremstyle{remark}
\newtheorem{remark}[theorem]{Remark}

\numberwithin{equation}{section}

\DeclareMathOperator{\supp}{supp}
\DeclareMathOperator{\loc}{loc}
\DeclareMathOperator{\Ric}{Ric}

\def\RR{\mathbb{R}}

\title[Semilinear wave inequalities on Riemannian Manifolds]{Semilinear wave inequalities with double damping and potential terms on Riemannian Manifolds}

\author[M. Jleli]{Mohamed Jleli}
\address{(M. Jleli) Department of Mathematics, College of Science, King Saud University, Riyadh 11451, Saudi Arabia}
\curraddr{}
\email{jleli@ksu.edu.sa}
\thanks{}

\author[M. Ruzhansky]{Michael Ruzhansky}
\address{(M. Ruzhansky) Department of Mathematics: Analysis, Logic and Discrete Mathematics, Ghent University, Belgium\newline 
School of Mathematical Sciences, Queen Mary University of London, United Kingdom}
\email{michael.ruzhansky@ugent.be}
\thanks{}

\author[B. Samet]{Bessem Samet}
\address{(B. Samet) Department of Mathematics, College of Science, King Saud University, Riyadh 11451, Saudi Arabia}
\curraddr{}
\email{bsamet@ksu.edu.sa}
\thanks{}

\author[B.T. Torebek]{Berikbol T. Torebek}
\address{(B.T. Torebek) Department of Mathematics: Analysis, Logic and Discrete Mathematics, Ghent University, Ghent, Belgium\newline 
Institute of Mathematics and Mathematical Modeling, Almaty, Kazakhstan}
\email{berikbol.torebek@ugent.be}
\thanks{}
\date{}

\begin{document}



\date{}

\dedicatory{}

\begin{abstract}
We study a semilinear wave inequality with double damping   on a complete noncompact Riemannian manifold.  The considered problem involves a potential  function $V$ depending on the space variable in front of the power nonlinearity and an inhomogeneous term $W$ depending on both time and space variables. Namely, we establish sufficient conditions for the nonexistence of weak solutions in both cases: $W\equiv 0$ and $W\not\equiv 0$. The obtained conditions depend on the parameters of the problem as well as the geometry of the manifold. Some special cases of manifolds, and of $V$ and $W$ are discussed in detail. 
\end{abstract}

\maketitle

\tableofcontents

\section{Introduction}\label{sec1}

Let  $\mathbb{M}$ be a complete noncompact Riemannian manifold of dimension $N$ equipped with a metric $g=(\cdot,\cdot)$ and a Riemannian distance $d(\cdot,\cdot)$. Our purpose is to investigate the nonexistence of a weak solution to the double damping semilinear wave inequality

\begin{eqnarray}\label{P1}
\left\{\begin{array}{llll}
u_{tt}-\Delta u+u_t-\Delta  u_t \geq  V(x)|u|^p +W(t,x), \,\,t>0,\,x\in \mathbb{M}, \\[10pt]
(u(0,x),u_t(0,x))=(u_0(x),u_1(x)),\,\,x\in \mathbb{M},	
\end{array}
\right.	
\end{eqnarray}
where $u=u(t,x)$, $\Delta$ is the Laplace-Beltrami operator, $p>1$, $V$ is a measurable function, $V>0$ a.e. (almost everywhere) in $\mathbb{M}$, $W\in L^1_{\loc}([0,\infty)\times \mathbb{M})$, $W\geq 0$ a.e. in $(0,\infty)\times \mathbb{M}$ and $u_i\in L^1_{\loc}(\mathbb{M})$, $i=0,1$.  The cases $W\equiv 0$ and $W\not\equiv 0$ are studied separately.  

The analysis of semilinear damped wave equations on  $\RR^N$ has a long history.  In \cite{TO}, the semilinear wave equation with frictional damping
\begin{equation}\label{P1-Ref}
u_{tt}-\Delta u+u_t	=|u|^p,\quad t>0,\,x\in \RR^N,
\end{equation}
where $p>1$, has been investigated. It was shown that $1+\frac{2}{N}$ is a critical exponent for \eqref{P1-Ref} in the following sense: 
\begin{itemize}
\item[(a)] 	If $\int_{\RR^N}u(0,x)\,dx>0$, $\int_{\RR^N}u_t(0,x)\,dx>0$ and $p<1+\frac{2}{N}$, then \eqref{P1-Ref} possesses no global solution;
\item[(b)] If $1+\frac{2}{N}<p<\frac{N}{N-2}$, $N\geq 3$ or $p>1+\frac{2}{N}$, $N\in\{1,2\}$,  then \eqref{P1-Ref} admits a unique global solution for suitable initial values. 
\end{itemize}
It was shown later (see \cite{Zhang 2001}) that the critical exponent $1+\frac{2}{N}$ belongs to case (a). We recall that $1+\frac{2}{N}$ is also critical for the semilinear parabolic equation $u_t-\Delta u=|u|^p$, $p>1$ (see Fujita \cite{Fujita}). In \cite{IK,IK2}, the authors considered the semilinear wave equation with double damping
\begin{equation}\label{P1-Ref2}
u_{tt}-\Delta u+u_t-\Delta u_t=|u|^p, \quad t>0,\,x\in \RR^N,
\end{equation}
where $p>1$.  It was proven that $1+\frac{2}{N}$ is also critical for \eqref{P1-Ref2} in the case $N\in\{1,2,3\}$.  In \cite{DA}, it was shown that under suitable conditions on the initial data, $1+\frac{2}{N}$ is critical for   \eqref{P1-Ref2} for every dimension $N$. In \cite{JSV21}, the authors considered \eqref{P1-Ref2} with a potential function $V$ and an inhomogeneous term $W$, namely
\begin{equation}\label{P1-Ref3}
u_{tt}-\Delta u+u_t-\Delta u_t=V(x)|u|^p+W(t,x),  \quad t>0,\,x\in \RR^N,
\end{equation}
where $p>1$,  $V$ is a measurable function, $V>0$ a.e.  in $\mathbb{R}^N$, $W\in L^1_{\loc}((0,\infty)\times \mathbb{R}^N)$, $W\geq 0$ a.e. in $(0,\infty)\times \mathbb{R}^N$ and  $W\not\equiv 0$. Namely, a general nonexistence theorem has been established.  In particular, if $V(x)=(1+|x|^2)^{\frac{\alpha}{2}}$ and $W=w(x)$, where $\alpha>-2$ and $w\geq 0$ is a nontrivial $L^1_{\loc}(\RR^N)$-function, it was proven that, if $N=2$, then \eqref{P1-Ref3} admits no global weak solution; if $N\geq 3$, then $\frac{N+\alpha}{N-2}$ is critical for \eqref{P1-Ref3}.  It is interesting to observe that when $N\geq 3$, the critical exponent of \eqref{P1-Ref2} jumps from $1+\frac{2}{N}$ to $1+\frac{2}{N-2}$ (the critical exponent for \eqref{P1-Ref3} with $V\equiv 1$ ($\alpha=0$) and  $W=w(x)$). Other references related to nonlinear damped wave equations on $\RR^N$ can be found in \cite{HA1,HA2,HO,IKeda,IK3,Kirane} (see also the references therein).

The issue of nonexistence for hyperbolic inequalities on $\RR^N$ has been investigated in several papers. For instance, in \cite{MI} (see also  \cite{PO}), the authors studied the semilinear wave inequality
\begin{equation}\label{P1-Ref4}
u_{tt}-\Delta u\geq |u|^p, 	\quad t>0,\,x\in \RR^N,
\end{equation}
where $p>1$ and $N\geq 2$.  Namely, it was shown that, if 
$$
\liminf_{R\to \infty} \int_{|x|<R}u_t(0,x)\,dx\geq 0,
$$ 
then for all 
\begin{equation}\label{Kato}
p\leq \frac{N+1}{N-1}
\end{equation}
\eqref{P1-Ref4} possesses no nontrivial weak solution. Furthermore, the authors pointed out the sharpness of condition \eqref{Kato}. For other results related to hyperbolic inequalities on $\RR^N$, see e.g. \cite{FI,JS2,LA,MI2}. 

The issue of nonexistence for evolution equations and inequalities have been also studied  when the Euclidean space is replaced by  a Heisenberg or graded Lie groups \cite{BO-22,DAM,Fino,FIR,JKS,KA,PO2,RU11,RU01,RU3,Zhang0} and a  hyperbolic space  \cite{BA,Punzo,TAT,Wang,Wang2}. 

Some extensions of nonexistence theorems when the Euclidean space is replaced by a noncompact Riemannian  manifold have been obtained in \cite{GR-K,GR-Sun,JSV,MA,MPS,Zhang1,Zhang2}. For instance, in \cite{Zhang2}, the author considered the inhomogeneous semilinear wave equation
\begin{equation}\label{P-Zhang}
u_{tt}-\Delta u=|u|^p+W(x),\quad t>0,\, x\in \mathbb{M},	
\end{equation}
where $\mathbb{M}$ is a noncompact complete manifold of dimension $N\geq 3$, $p>1$,  $W\geq 0$ is $L^1_{\loc}$-function and $W\not\equiv 0$.   The considered manifold $\mathbb{M}$ was supposed to satisfy the following conditions:
\begin{itemize}
\item[(C$_1$)] 	The Ricci curvature of $\mathbb{M}$ is nonnegative;
\item[(C$_2$)] There exists a point $x_0\in  \mathbb{M}$ such that the cut locus to it is empty;
\item[(C$_3$)] $\left|\partial \log g^{1/2}/\partial r\right|\leq C/r$, where $g^{1/2}$ is the volume density and $r$ is the distance from $x_0\in \mathbb{M}$;
\item[(C$_4$)] There exists a nonnegative function $f$ such that 
$$
C^{-1}f(t)\leq |B(x,t)|\leq Cf(t),\quad t>0,
$$ 
where $C>0$ is constant, $f(t)\sim t^\alpha$ as $t\to \infty$, $\alpha>2$, and $f(t)\sim t^N$ as $t\to 0^+$. 
\end{itemize}
Under the above conditions, it was proven that  $\frac{\alpha}{\alpha-2}$ is the critical exponent for \eqref{P-Zhang}. More precisely, it was shown that, 
\begin{itemize}
\item[(i)] if $1<p<\frac{\alpha}{\alpha-2}$, then \eqref{P-Zhang} possesses no
global solution for any initial data;
\item[(ii)] if $p>\frac{\alpha}{\alpha-2}$, then \eqref{P-Zhang} admits global solutions for some $w>0$. 
\end{itemize}
Very recently, it was proven in \cite{JSV} that the critical exponent  $\frac{\alpha}{\alpha-2}$  belongs to the case (i).  

Notice that the considered assumptions on the manifold $\mathbb{M}$ (in particular condition (C$_3$)) are  somewhat restrictive.  In \cite{GR-K,GR-Sun,MA}, the used argument for the proofs of the  nonexistence theorems is based on the nonlinear capacity method (see e.g. \cite{MI2}) with an appropriate choice of radial test functions, i.e. functions depending on $r(x)$. Notice that in the weak formulations of the considered problems, only the gradient of such test functions was involved, but not their Laplacian. This leads to the consideration of a general manifold $\mathbb{M}$ since the gradient of $r(x)$ is well-defined.  In our case, due to the presence of the term $u_{tt}$ in \eqref{P1} and the fact that no restriction is imposed on the sign of $u$, the above technique cannot be used. Namely, the weak formulation of \eqref{P1} involves $\Delta \psi$, where $\psi$ belongs to a certain class of test functions. Consequently, we need a family of smooth test functions with controlled gradient and Laplacian.  To overcome this difficulty, we make use of a  result due to Bianchi and  Setti \cite{BI}, where a family of cut-off functions with controlled gradient and Laplacian on manifolds with Ricci curvature bounded from below by a (possibly unbounded) nonpositive function of the distance from a fixed reference point has been constructed. We refer to \cite{JSS,JS3,JSV22}, where Bianchi-Setti result has been used to establish nonexistence theorems for some evolution inequalities. 

In this paper, we first consider \eqref{P1} with $W\equiv 0$. Under certain assumptions on the Ricci curvature of the manifold, we establish sufficient conditions for the nonexistence of weak solutions. We next consider the inhomogeneous case ($W(t,x)\not\equiv 0$) and study the effect of the term $W$ on the obtained nonexistence results  in the homogeneous case ($W\equiv 0$). 

The organization of this paper is as follows. In Section \ref{sec2}, we recall some notions and results from Riemannian geometry. In Section \ref{sec3}, after defining weak solutions to \eqref{P1}, we state our main results and study some particular cases of manifolds, potential functions $V$ and inhomogeneous terms $W$. Finally, Section \ref{sec4} is devoted to the proofs of our main results.

We end this section by fixing some notations. Throughout this paper, the symbols $C$ or  $C_i$
 denote always generic positive constants, which are independent of the scaling parameters $T$, $R$ and the solution $u$. Their values could be changed from one line to another. The notation $R\gg 1$ means that  $R>0$ is sufficiently large. We will use the notation $h\sim k$
 for two positive functions or quantities, which satisfy $C_1 h\leq k\leq C_2h$. 

\section{Preliminaries on Riemannian geometry} \label{sec2}

This section is devoted to some notions and results from Riemannian geometry. For more details, the reader is referred to \cite{ALI,GRI99,Peter}. 

Let $\mathbb{M}$ be a Riemannian manifold of dimension $N$, equipped with a metric $g=(\cdot,\cdot)$ and a Riemannian distance $d(\cdot,\cdot)$. Let $(U,\phi)$ be a chart in $\mathbb M$, where  $\phi: U\to \phi(U)\subset \mathbb{R}^N$ and 
$$
\phi(x) =(x^1(x),x^2(x),\cdots,x^N(x)),\quad x\in U.
$$
Let $\partial_i$, $i=1,2,\cdots,N$,  be  the corresponding vector fields on $U$. The local components of the metric $g$ are defined by
$$
g_{ij}=(\partial_i,\partial_j),\quad i,j=1,2,\cdots,N.
$$ 
The inverse of the matrix $(g_{ij})$ is denoted by $(g^{ij})$. For a smooth function $u: \mathbb M\to \mathbb{R}$, the gradient of $u$, $\nabla u$, is the vector field defined by  
$$
(\nabla u,X)=Xu
$$
for each vector field $X$ on $\mathbb M$.  In local coordinates, we have
$$
\nabla u=\sum_{i,j} g^{ij}(\partial_iu) \partial_j. 
$$
The divergence of  a vector field $X$ on $\mathbb M$ is defined by 
$$
\mbox{div}(X)(x)=\mbox{trace}\left(T_x\mathbb M \ni \xi \mapsto {\widetilde{\nabla}}_\xi X\in T_x\mathbb M\right),\quad x\in \mathbb M,
$$
where $T_x \mathbb M$ is the tangent vector space at $x\in \mathbb M$ and $\widetilde{\nabla}$ is the  Levi-Civita connection associated to the metric $g$. In local coordinates, if $X=\displaystyle \sum_{i}X_i\partial_i$, the divergence of $X$ can be expressed as
$$
\mbox{div}(X)=\frac{1}{\sqrt{\mathbf{g}}} \sum_j \partial_j\left(\sqrt{\mathbf{g}}X_j\right),
$$
where $\mathbf{g}=\det(g_{ij})$.  The Laplacian of $u$ is the function
$$
\Delta u= \mbox{div}(\nabla u).
$$
In local coordinates,  we have
$$
\Delta u=\frac{1}{\sqrt{\mathbf{g}}} \sum_{i,j} \partial_i\left(\sqrt{\mathbf{g}}g^{ij}\partial_j u\right).
$$
We denote by $\Ric$ the Ricci tensor which is expressed in local coordinates as
$$
R_{ij}=R_{ji}=\sum_{\ell} \partial_\ell\left(\Gamma_{ij}^\ell\right)-\sum_{\ell} \partial_j\left(\Gamma_{i\ell}^{\ell}\right)+\sum_{k,\ell} \left(\Gamma_{ij}^k\Gamma_{k\ell}^{\ell}-\Gamma_{i\ell}^k\Gamma_{kj}^\ell\right),
$$
where $\Gamma_{ij}^k$ are the Christoffel symbols. For a given function $f: \mathbb M\to \mathbb{R}$, the notation $\Ric\geq f(x)$ means that
$$
\Ric(X,X)\geq f(x) |X|^2
$$
for every vector field $X$ on $\mathbb  M$, where $|X|=\sqrt{(X,X)}$. 

For $\delta>0$, the geodesic ball with center $x_0\in \mathbb{M}$ and radius $\delta$, is denoted
by $B(x_0,\delta)$, that is, 
$$
B(x_0,\delta)=\left\{x\in \mathbb{M}: d(x_0,x)<\delta\right\}.
$$
The volume of $B(x_0,\delta)$ is denoted by $\left|B(x_0,\delta)\right|$.

We recall below a very important result due to  Bianchi and Setti (see \cite{BI}). 

\begin{lemma}\label{L2.1}{\rm{
	Let $x_0\in \mathbb{M}$ and suppose that for some $C_0\geq 0$ and $\sigma\in [-2,2]$,  we have   
	\begin{equation}\label{Ricc1} 
		\Ric\geq -C_0 (N-1) 	\left(1+d^2(x_0,x)\right)^{\frac{-\sigma}{2}}.
	\end{equation}
	Then,  there exists a family of functions  $\{\psi_R\}_{R> 1}\subset C_c^\infty(\mathbb M)$ satisfying the following:
	\begin{enumerate}
		\item[\rm{(i)}] $0\leq \psi_R\leq 1$, ${\psi_R}{|_{B(x_0,R)}}\equiv 1$;
		\item[\rm{(ii)}] there exists $\gamma_0=\gamma_0(\mathbb{M})>1$ (independent of $R$) such that 
		$$
		\supp(\psi_R)\subset B(x_0,\gamma R),\quad \gamma>\gamma_0;
		$$ 
		\item[\rm{(iii)}] $|\nabla \psi_R|\leq CR^{-1}$;
		\item[\rm{(iv)}] $|\Delta \psi_R|\leq CR^{-\left(1+\frac{\sigma}{2}\right)}$.
	\end{enumerate}
	Here, by $\psi_R\in C_c^\infty(\mathbb{M})$, we mean that   $\psi_R\in C^\infty(\mathbb{M})$ and $\supp(\psi_R)$ is compact in $\mathbb{M}$.}}
\end{lemma}

\begin{remark}\label{RK2.2}{\rm{
	Remark that, if  \eqref{Ricc1} holds for some $\sigma>2$,  	then it holds for $\sigma=2$. Hence, by Lemma \ref{L2.1}, if \eqref{Ricc1} is satisfied for some $x_0\in \mathbb{M}$, $C_0\geq 0$ and $\sigma\geq -2$, then  there exists a family of functions  $\{\psi_R\}_{R> 1}\subset C_c^\infty(\mathbb M)$  satisfying \rm{(i)--(iii)} and 
	\begin{equation}\label{pptiv}
		|\Delta \psi_R|\leq CR^{-\left(1+\frac{\min\{\sigma,2\}}{2}\right)}.
	\end{equation}
	}}
\end{remark}

The proofs of the following results can be found in \cite{JSS}. 

\begin{lemma}\label{LL2.3}
{\rm{Let $p>1$ and $x_0\in \mathbb{M}$ be such that 
$$
|B(x_0,R)|\leq a R^{b}
$$
for all $R>0$, where  $a, b>0$ are constants. Let $V$ be a measurable function such that $V^{\frac{-1}{p-1}}\in L^1_{\loc}(\mathbb{M})$ and 
$$
V(x)\geq C d^\kappa(x_0,x),\quad \mbox{a.e. in } \mathbb{M},	
$$
where $C>0$ and $\kappa\in \mathbb{R}$ are constants.  Then, we have as $R\to \infty$, 
$$
\int_{B(x_0,R)} V^{\frac{-1}{p-1}}(x)\,d\mu =O\left(\ln R+R^{\omega+b}\right),
$$ 	
where $\displaystyle \omega=-\frac{\kappa}{p-1}$.}}
\end{lemma}

\begin{lemma}\label{LL2.4}{\rm{ Let $p>1$ and $x_0\in \mathbb{M}$ be such that 
$$
|B(x_0,R)|\leq \alpha \exp\left(\beta R^{\gamma}\right)
$$
for all $R>0$, where $\alpha,\beta,\gamma>0$ are constants. Let $V$ be a measurable function such that $V^{\frac{-1}{p-1}}\in L^1_{\loc}(\mathbb{M})$ and 
$$
V(x)\geq C d^{\kappa_1}(x_0,x)\exp\left(\kappa_2 d^{\gamma}(x_0,x)\right),\quad \mbox{a.e. in } \mathbb{M},	
$$
where  $C>0$ and $\kappa_1,\kappa_2\in \mathbb{R}$ are constants.  Assume that 
$$
\kappa_2\geq 2^{\gamma} \beta (p-1).
$$
Then, we have as $R\to \infty$, 
$$
\int_{B(x_0,R)} V^{\frac{-1}{p-1}}\,d\mu=O\left(\ln R+R^\omega\right),  
$$
where $\displaystyle \omega=-\frac{\kappa_1}{p-1}$. }}	
\end{lemma}

\section{Statement of the results and discussions}\label{sec3}

We first define weak solutions to \eqref{P1}.  Let $Q=[0,\infty)\times \mathbb{M}$. We introduce the set
$$
\Psi=\left\{\psi\in C^2(Q): \psi\geq 0,\, \supp(\psi)\subset\subset Q\right\}.
$$

\begin{definition}\label{def-ws1}{\rm{
Let $p>1$, $V=V(x)>0$ a.e. in $\mathbb{M}$, $W=W(t,x)\in L^1_{\loc}(Q)$ and $u_i\in L^1_{\loc}(\mathbb{M})$, $i=0,1$. We say that $u\in L^p_{\loc}\left(Q,V\,dx\,dt\right)\cap L^1_{\loc}(Q)$ is a weak solution to \eqref{P1}, if    
\begin{eqnarray}\label{ws1}
\nonumber &&\int_{Q}|u|^p \psi V\,dx\,dt +\int_{\mathbb{M}}u_0(x)\left(\psi(0,x)-\psi_t(0,x)-\Delta \psi(0,x)\right)\,dx\\
\nonumber &&+\int_{\mathbb{M}}u_1(x)\psi(0,x)\,dx+\int_Q W\psi\,dx\,dt\\
&&\leq \int_{Q}u\left(\psi_{tt}-\Delta \psi-\psi_t+\Delta \psi_t\right)\,dx\,dt
\end{eqnarray}
 for every $\psi\in \Psi$. 	}}
\end{definition}

Notice that any smooth solution to \eqref{P1} is a weak solution in the sense of Definition \ref{def-ws1}. This can be easily seen by multiplying the inequality in \eqref{P1} by $\psi\in \Psi$ and integrating by parts over $Q$.

We first consider the case $W\equiv 0$. Our first main result  is stated in the following theorem.

\begin{theorem}\label{T3.2}{\rm{
Let  $x_0\in \mathbb{M}$ be  such that \eqref{Ricc1} holds for some $C_0\geq 0$ and $\sigma> -2$. 
Let $p>1$, $V=V(x)>0$ a.e. in $\mathbb{M}$ and $V^{\frac{-1}{p-1}}\in L^1_{\loc}(\mathbb{M})$. Suppose that $u_i\in L^1(\mathbb{M})$, $i=0,1$, and
\begin{equation}\label{cdu01}
\sum_{i=0}^1 \int_{\mathbb{M}}u_i(x)\,dx>0. 	
\end{equation}
Assume that 
\begin{equation}\label{cd-blow-up-P1}
\liminf_{R\to \infty} R^{-\frac{\left(1+\frac{\min\{\sigma,2\}}{2}\right)}{p-1}}\int_{B(x_0,R)}V^{\frac{-1}{p-1}}(x)\,dx=0.
\end{equation}
Then \eqref{P1} possesses no weak solution.
}}
\end{theorem}

We discuss below some particular cases of Theorem \ref{T3.2}.
 
 Assume first that $V^{\frac{-1}{p-1}}\in L^1(\mathbb{M})$.   In this case, by the dominated convergence theorem, we have
$$
\lim_{R\to \infty} \int_{B(x_0,R)}		V^{\frac{-1}{p-1}}(x)\,dx=\int_{\mathbb{M}}V^{\frac{-1}{p-1}}(x)\,dx<\infty. 
$$
Hence, \eqref{cd-blow-up-P1} holds for all $p>1$, and Theorem \ref{T3.2} applies.  Consequently, we have the following result.

\begin{corollary}\label{CR3.3}{\rm{
Let  $x_0\in \mathbb{M}$ be  such that \eqref{Ricc1} holds for some $C_0\geq 0$ and $\sigma> -2$. 
Let $p>1$, $V=V(x)>0$ a.e. in $\mathbb{M}$ and $V^{\frac{-1}{p-1}}\in L^1(\mathbb{M})$. Suppose that $u_i\in L^1(\mathbb{M})$, $i=0,1$, and \eqref{cdu01} holds. Then \eqref{P1} possesses no weak solution.	}}
\end{corollary}

Assume now that \eqref{Ricc1} holds for some $C_0\geq 0$ and $\sigma\geq 2$. In this case, by volume comparison theorems (see Proposition 2.11 and Theorem 2.14 in \cite{PI}), we have 
\begin{equation}\label{vol-ba3d}
|B(x_0,R)|\leq C_{\mathbb{M}}R^\nu,\quad R>0,
\end{equation}
where $C_{\mathbb{M}}>0$ is a constant and 
\begin{equation}\label{nu-f}
\nu=(N-1)\delta_\sigma+1,\quad \delta_\sigma=\left\{\begin{array}{llll}
		\frac{1+\sqrt{1+4C_0}}{2}&\mbox{if}& \sigma=2,\\
		1&\mbox{if}& \sigma>2.	
	\end{array}
	\right.
\end{equation}
We consider potential functions $V$ satisfying $V^{\frac{-1}{p-1}}\in L^1_{\loc}(\mathbb{M})$ and 
\begin{equation}\label{V-case2}
V(x)\geq C_V d^\kappa(x_0,x),\quad \mbox{a.e. in }\mathbb{M},	
\end{equation}
where $C_V>0$ and $\kappa>-2$ are constants. Then, by Lemma \ref{LL2.3},  we have as $R\to \infty$, 
$$
\int_{B(x_0,R)}V^{\frac{-1}{p-1}}(x)\,dx=O\left(\ln R+R^{\nu-\frac{\kappa}{p-1}}\right),
$$
which yields
\begin{equation}\label{LAM}
\begin{aligned}
R^{-\frac{\left(1+\frac{\min\{\sigma,2\}}{2}\right)}{p-1}}\int_{B(x_0,R)}V^{\frac{-1}{p-1}}(x)\,dx&=R^{\frac{-2}{p-1}}\int_{B(x_0,R)}V^{\frac{-1}{p-1}}(x)\,dx\\
&=O\left(R^{\frac{-2}{p-1}}\ln R+R^{\nu-\frac{\kappa+2}{p-1}}\right).
\end{aligned}
\end{equation}
 Hence, if 
$$
\nu-\frac{\kappa+2}{p-1}<0,
$$
that is,
$$
1<p<1+\frac{\kappa+2}{\nu},
$$
then \eqref{cd-blow-up-P1} holds and Theorem \ref{T3.2} applies. Consequently, we have the following result.

\begin{corollary}\label{CR3.4}{\rm{
Let  $x_0\in \mathbb{M}$ be such that \eqref{Ricc1} holds for some $C_0\geq 0$ and $\sigma\geq 2$. Let $p>1$,  $V^{\frac{-1}{p-1}}\in L^1_{\loc}(\mathbb{M})$ and $V$ satisfies \eqref{V-case2} for some constants $C_V>0$ and $\kappa>-2$.    Suppose that $u_i\in L^1(\mathbb{M})$, $i=0,1$, and \eqref{cdu01} holds. If 
\begin{equation}\label{pknu}
p<p^*(\kappa,\nu):=1+\frac{\kappa+2}{\nu},
\end{equation}
where $\nu$ is given by \eqref{nu-f}, then  \eqref{P1} possesses no weak solution.	}}
\end{corollary}

Consider now the Euclidean case $\mathbb{M}=\mathbb{R}^N$. In this case, \eqref{Ricc1} holds for $x_0=0$ (the origin), $C_0=0$ and $\sigma=2$.  Furthermore, \eqref{vol-ba3d} holds with $\nu=N$. Hence, from Corollary \ref{CR3.4}, we deduce the following result.

\begin{corollary}\label{CR3.5}{\rm{
Let $p>1$, $V^{\frac{-1}{p-1}}\in L^1_{\loc}(\mathbb{R}^N)$ and 
$$
V(x)\geq C_V |x|^\kappa, \quad  \mbox{a.e. in }\mathbb{R}^N,
$$
where $C_V>0$ and $\kappa>-2$ are constants. Suppose that $u_i\in L^1(\mathbb{R}^N)$, $i=0,1$, and
$$
\sum_{i=0}^1 \int_{\mathbb{R}^N}u_i(x)\,dx>0. 	
$$
If 
$$
p<p^*(\kappa,N):=1+\frac{\kappa+2}{N},
$$
then  \eqref{P1} possesses no weak solution.	}}
\end{corollary}

We next consider the case when  \eqref{Ricc1} holds for some $C_0\geq 0$ and $-2< \sigma<2$.  In this case, by volume comparison theorems (see e.g. \cite{GRI99}), we have
\begin{equation}\label{VOL-end}
|B(x_0,R)|\leq C_{\mathbb{M}}\exp\left(\tau(N-1)R^{1-\frac{\sigma}{2}}\right),\quad R>0,
\end{equation}
where  $C_{\mathbb{M}},\tau>0$ are constants. We consider potential functions $V$ satisfying 
$V^{\frac{-1}{p-1}}\in L^1_{\loc}(\mathbb{M})$ and 
\begin{equation}\label{V-endcase}
V(x)\geq C_V d^a(x_0,x) \exp\left(\rho \, d^{1-\frac{\sigma}{2}}(x_0,x)\right), \quad \mbox{a.e. in }\mathbb{M},	
\end{equation}
where $C_V>0$, $a\in \mathbb{R}$ and $\rho\geq 2^{1-\frac{\sigma}{2}}\tau(N-1)(p-1)$ are constants. 
In this case, by \eqref{VOL-end}, \eqref{V-endcase} and Lemma \ref{LL2.4}, we have  as $R\to \infty$, 
$$
\int_{B(x_0,R)}V^{\frac{-1}{p-1}}(x)\,dx=O\left(\ln R+R^{\frac{-a}{p-1}}\right),
$$
which yields
$$
\begin{aligned}
R^{-\frac{\left(1+\frac{\min\{\sigma,2\}}{2}\right)}{p-1}}\int_{B(x_0,R)}V^{\frac{-1}{p-1}}(x)\,dx&=R^{-\frac{\left(1+\frac{\sigma}{2}\right)}{p-1}}\int_{B(x_0,R)}V^{\frac{-1}{p-1}}(x)\,dx\\
&=O\left(R^{-\frac{1+\frac{\sigma}{2}}{p-1}}\ln R+R^{-\frac{1+\frac{\sigma}{2}+a}{p-1}}\right).
\end{aligned}
$$
So, if 
$$
1+\frac{\sigma}{2}+a>0,
$$
then \eqref{cd-blow-up-P1} holds and Theorem \ref{T3.2} applies. Therefore, we deduce the following result. 

\begin{corollary}\label{CR.3.7}{\rm{
 Let $x_0\in \mathbb{M}$ be such that \eqref{Ricc1} holds for some $C_0\geq 0$ and $-2<\sigma<2$.  Let $p>1$,  $V^{\frac{-1}{p-1}}\in L^1_{\loc}(\mathbb{M})$ and $V$ satisfies \eqref{V-endcase} for some constants $C_V>0$, $a>-\left(1+\frac{\sigma}{2}\right)$ and $\rho\geq 2^{1-\frac{\sigma}{2}}\tau(N-1)(p-1)$, where $\tau>0$ is given by \eqref{VOL-end}. If $u_i\in L^1(\mathbb{M})$, $i=0,1$, and \eqref{cdu01} holds, then \eqref{P1} possesses no weak solution.	}}
\end{corollary}

We next consider the case $W\geq 0$ a.e. in $Q$ and $W\not\equiv 0$. Our second main result is stated below.

\begin{theorem}\label{T3.9}{\rm{
Let $x_0\in \mathbb{M}$ be such that \eqref{Ricc1} holds for some $C_0\geq 0$ and $\sigma> -2$.
Let $p>1$, $V=V(x)>0$ a.e. in $\mathbb{M}$, $V^{\frac{-1}{p-1}}\in L^1_{\loc}(\mathbb{M})$, $W=W(t,x)\in L^1_{\loc}(Q)$, $W\geq 0$ a.e. in $Q$ and $W\not\equiv 0$.   Let $u_0\in L^1(\mathbb{M})$, $u_1\in L^1_{\loc}(\mathbb{M})$ and 
\begin{equation}\label{u01-cdP2}
(u_0+u_1)(x)\geq 0,\quad \mbox{a.e. in }\mathbb{M}.	
\end{equation}
Assume that there exists $\gamma>\gamma_0$ ($\gamma_0>1$ is provided by Lemma \ref{L2.1}) such that 
\begin{equation}\label{cd2-blowup-P2}
\liminf_{R\to \infty} R^{-\frac{\left(1+\frac{\min\{\sigma,2\}}{2}\right)}{p-1}}\left(\int_{B(x_0,\gamma R)} V^{\frac{-1}{p-1}}(x)\,dx\right)\left(\int_{0}^{\frac{R^{1+\frac{\min\{\sigma,2\}}{2}}}{2}}\int_{B(x_0,R)} W(t,x)\,dx\,dt\right)^{-1}=0.
\end{equation}
Then \eqref{P1} possesses no weak solution.
}}
\end{theorem}

\begin{remark}\label{RK3.9}{\rm{
It is interesting to observe that, if $W\in L^1(Q)$, then conditions \eqref{cd-blow-up-P1} and \eqref{cd2-blowup-P2} are equivalent, which means that the inhomogeneous term $W$ has no effect on the nonexistence of a weak solution. }}
\end{remark}

We discuss below some particular cases of Theorem \ref{T3.9}.  Let 
$$
W(t,x)=f(t)g(x),\quad \mbox{a.e. in }Q, 
$$
where $f\in L^1_{\loc}([0,\infty))$, $f\geq 0$ a.e. in $(0,\infty)$, $f\not\equiv 0$,  $g\in L^1_{\loc}(\mathbb{M})$, $g\geq 0$ a.e. in $\mathbb{M}$ and $g\not\equiv 0$.  Assume that we have as $R\to \infty$, 
\begin{equation}\label{cdWW}
\left(\int_0^R f(t)\,dt\right)^{-1}\sim	R^\lambda
\end{equation}
and
\begin{equation}\label{cdWW2}
\left(\int_{B(x_0,R)} g(x)\,dx\right)^{-1}\sim	R^\eta,
\end{equation}
where $\lambda,\eta\in \RR$ are constants.  This implies that 
\begin{equation}\label{cdWW3}
\left(\int_{0}^{\frac{R^{1+\frac{\min\{\sigma,2\}}{2}}}{2}}\int_{B(x_0,R)} W(t,x)\,dx\,dt\right)^{-1}\sim	 R^{\lambda\left(1+\frac{\min\{\sigma,2\}}{2}\right)+\eta}.
\end{equation}
Notice that, if $$\lambda\left(1+\frac{\min\{\sigma,2\}}{2}\right)+\eta> 0,$$ then $W \equiv 0.$
If $$\lambda\left(1+\frac{\min\{\sigma,2\}}{2}\right)+\eta=0,$$ then $W\in L^1(Q).$ In this case, by Remark \ref{RK3.9}, conditions \eqref{cd-blow-up-P1} and \eqref{cd2-blowup-P2} are equivalent. So, we will only discuss the case
\begin{equation}\label{MOUHIM}
\lambda\left(1+\frac{\min\{\sigma,2\}}{2}\right)+\eta<0. 
\end{equation}

Assume now that \eqref{Ricc1} holds for some $C_0\geq 0$ and $\sigma\geq 2$. We consider potential functions $V$ satisfying $V^{\frac{-1}{p-1}}\in L^1_{\loc}(\mathbb{M})$ and 
\eqref{V-case2}, where $C_V>0$ and $\kappa\in \RR$ are constants. In view of \eqref{LAM} and \eqref{cdWW3},  for all $\gamma>\gamma_0$, we obtain as $R\to \infty$,
$$
\begin{aligned}
&R^{-\frac{\left(1+\frac{\min\{\sigma,2\}}{2}\right)}{p-1}}\left(\int_{B(x_0,\gamma R)} V^{\frac{-1}{p-1}}(x)\,dx\right)\left(\int_{0}^{\frac{R^{1+\frac{\min\{\sigma,2\}}{2}}}{2}}\int_{B(x_0,R)} W(t,x)\,dx\,dt\right)^{-1}\\
&=O\left(R^{\frac{-2}{p-1}+2\lambda+\eta}\ln R+R^{2\lambda+\eta+\nu-\frac{\kappa+2}{p-1}}\right),
\end{aligned}
$$
where $\nu$ is given by \eqref{nu-f}.  Then, by \eqref{MOUHIM}, we deduce that \eqref{cd2-blowup-P2} holds provided that 
$$
2\lambda+\eta+\nu-\frac{\kappa+2}{p-1}<0,
$$
that is,
$$
(2\lambda+\eta+\nu)(p-1)<\kappa+2. 
$$
Then, by Theorem \ref{T3.9}, we obtain the following result. 

\begin{corollary}\label{CR3.10}{\rm{
Let $x_0\in \mathbb{M}$ be such that \eqref{Ricc1} holds for some $C_0\geq 0$ and $\sigma\geq 2$.
Let $p>1$,  $V^{\frac{-1}{p-1}}\in L^1_{\loc}(\mathbb{M})$ and $V$ satisfies \eqref{V-case2} for some constants $C_V>0$ and $\kappa\in \RR$.  Let $$W(t,x)=f(t)g(x)\,\,\,\, \text{a.e. in} \,\,\,Q,$$ where  $f\in L^1_{\loc}([0,\infty))$, $f\geq 0$ a.e. in $(0,\infty)$, $f\not\equiv 0$,  $g\in L^1_{\loc}(\mathbb{M})$, $g\geq 0$ a.e. in $\mathbb{M}$ and $g\not\equiv 0$. Assume that $f$ and $g$ satisfy \eqref{cdWW} and \eqref{cdWW2} with
\begin{equation}\label{inter}
2\lambda+\eta<0,\quad (2\lambda+\eta+\nu)(p-1)<\kappa+2,
\end{equation}
where $\nu$ is given by \eqref{nu-f}. If $u_0\in L^1(\mathbb{M})$ and $u_1\in L^1_{\loc}(\mathbb{M})$ satisfy \eqref{u01-cdP2}, then \eqref{P1} possesses no weak solution. 
}}
\end{corollary}

\begin{remark}\label{RK3.11}{\rm{
Let us consider the special case of Corollary \ref{CR3.10}, where $\kappa>-2$. In this case, \eqref{inter} is equivalent to 
\begin{equation}\label{HAMD}
0<\nu\leq -(2\lambda+\eta)\quad\mbox{or}\quad \nu> -(2\lambda+\eta)>0,\,\, p<1+\frac{\kappa+2}{2\lambda+\eta+\nu}.
\end{equation} 
Setting
$$
p^*(\kappa,\nu,\lambda,\eta)=\left\{\begin{array}{llll}
\infty &\mbox{if}& 	0<\nu\leq -(2\lambda+\eta),\\
1+\displaystyle\frac{\kappa+2}{2\lambda+\eta+\nu} &\mbox{if}&  \nu> -(2\lambda+\eta)>0,
\end{array}
\right.
$$
\eqref{HAMD} can be written in the form
$$
p<p^*(\kappa,\nu,\lambda,\eta). 
$$
Observe that $$p^*(\kappa,\nu,\lambda,\eta)>p^*(\kappa,\nu),$$ where $p^*(\kappa,\nu)$ is given by \eqref{pknu} ($p^*(\kappa,\nu)$ corresponds to the homogeneous case $W\equiv 0$). This shows that the presence of the inhomogeneous term $W=W(t,x)$  leads to the enlargement of the range of $p$ for which \eqref{P1} possesses no weak solution.}}
\end{remark}

We now consider the special case of  Corollary \ref{CR3.10}, where $\mathbb{M}=\RR^N$. We recall that in this case, \eqref{Ricc1} holds for $x_0=0$ (the origin), $C_0=0$ and $\sigma=2$.  Furthermore, \eqref{vol-ba3d} holds with $\nu=N$. From Corollary \ref{CR3.10} and Remark \ref{RK3.11}, we deduce the following result. 

\begin{corollary}\label{CR3.12}{\rm{
Let $p>1$, $V^{\frac{-1}{p-1}}\in L^1_{\loc}(\mathbb{R}^N)$ and 
$$
V(x)\geq C_V |x|^\kappa, \quad  \mbox{a.e. in }\mathbb{R}^N,
$$
where $C_V>0$ and $\kappa>-2$ are constants. \\
Let $$W(t,x)=f(t)g(x)\,\,\,\, \text{a.e. in} \,\,\,\,Q,$$ where  $f\in L^1_{\loc}([0,\infty))$, $f\geq 0$ a.e. in $(0,\infty)$, $f\not\equiv 0$,  $g\in L^1_{\loc}(\RR^N)$, $g\geq 0$ a.e. in $\mathbb{R}^N$ and $g\not\equiv 0$. Assume that $f$ and $g$ satisfy \eqref{cdWW} and \eqref{cdWW2} with $2\lambda+\eta<0$. Let $u_0\in L^1(\mathbb{R}^N)$, $u_1\in L^1_{\loc}(\RR^N)$ and 
\begin{equation}\label{uu0uu1}
(u_0+u_1)(x)\geq 0,\quad \mbox{a.e. in }\mathbb{R}^N.	
\end{equation}
If $p<p^*(\kappa,N,\lambda,\eta)$, where 
$$
p^*(\kappa,N,\lambda,\eta)=\left\{\begin{array}{llll}
\infty &\mbox{if}& 	N\leq -(2\lambda+\eta),\\
1+\displaystyle\frac{\kappa+2}{2\lambda+\eta+N} &\mbox{if}&  N> -(2\lambda+\eta),
\end{array}
\right.
$$
then \eqref{P1} possesses no weak solution.
}}	
\end{corollary}

We next consider the particular case of Corollary \ref{CR3.12}, where $g\in L^1(\RR^N)$ and 
$$
f(t)=t^\zeta,\quad t>0
$$
for some constant $\zeta>-1$. In this case, for all $R>0$, we have
$$
\int_0^R f(t)\,dt=C R^{\zeta+1},
$$
 which shows that \eqref{cdWW} holds with $\lambda=-(\zeta+1)$.   Furthermore,  $g\in L^1(\RR^N)$ implies that \eqref{cdWW2} holds with $\eta=0$. Hence, from Corollary \ref{CR3.12}, we deduce the following result.

\begin{corollary}\label{CR3.13}{\rm{
Let $p>1$, $V^{\frac{-1}{p-1}}\in L^1_{\loc}(\mathbb{R}^N)$ and 
$$
V(x)\geq C_V |x|^\kappa, \quad  \mbox{a.e. in }\mathbb{R}^N,
$$
where $C_V>0$ and $\kappa>-2$ are constants. \\
Let $$W(t,x)=t^\zeta g(x)\,\,\,\, \text{a.e. in}\,\,\,\, Q,$$ where  $\zeta>-1$,  $g\in L^1(\RR^N)$, $g\geq 0$ a.e. in $\mathbb{R}^N$ and $g\not\equiv 0$.  Let $u_0\in L^1(\mathbb{R}^N)$ and $u_1\in L^1_{\loc}(\RR^N)$ satisfy \eqref{uu0uu1}.  
If $p<p^*(\kappa,N,\zeta)$, where 
$$
p^*(\kappa,N,\zeta)=\left\{\begin{array}{llll}
\infty &\mbox{if}& 	N\leq 2(\zeta+1),\\
1+\displaystyle\frac{\kappa+2}{N-2(\zeta+1)} &\mbox{if}&  N> 2(\zeta+1),
\end{array}
\right.
$$
then \eqref{P1} possesses no weak solution.
}}	
\end{corollary}

Taking $V\equiv 1$ and $\zeta=0$ in Corollary \ref{CR3.13} (so $\kappa=0$), we obtain the following result. 

\begin{corollary}\label{CR3.14}{\rm{
Let $V\equiv 1$ and $W=g\in L^1(\RR^N)$, $g\geq 0$ a.e. in $\mathbb{R}^N$ and $g\not\equiv 0$.
 Let $u_0\in L^1(\mathbb{R}^N)$ and $u_1\in L^1_{\loc}(\RR^N)$ satisfy \eqref{uu0uu1}.  
If
\begin{equation}\label{cd-S}
1<p<p^*(N):=\left\{\begin{array}{llll}
\infty &\mbox{if}& 	N\in \{1,2\},\\
1+\displaystyle\frac{2}{N-2} &\mbox{if}&  N\geq 3,
\end{array}
\right.
\end{equation}
then \eqref{P1} possesses no weak solution.}}	
\end{corollary}

\begin{remark}{\rm{Under the assumptions of Corollary \ref{CR3.14}, if $N\geq 3$ and $$p=p^*(N),$$ following the same argument used in \cite{JSV} (see also \cite{BO}), one can show that  \eqref{P1} possesses no weak solution. Furthermore, it is well known (see e.g. \cite{Zhang1}) that, if $N\geq 3$ and $$p>p^*(N),$$ then   \eqref{P1} possesses  stationary solutions for some $W=W(x)>0$. This shows the sharpness of the obtained condition \eqref{cd-S}.  
}}
\end{remark}

\section{Proofs of the main results}\label{sec4}

This section is devoted to the proofs of Theorems \ref{T3.2} and \ref{T3.9}. 

Assume that there exists $x_0\in \mathbb{M}$ such that \eqref{Ricc1} holds for some $C_0\geq 0$ and $\sigma>-2$.  Let $p>1$, $V=V(x)>0$ a.e. in $\mathbb{M}$, $W\in L^1_{\loc}([0,\infty)\times \mathbb{M})$ and $u_i\in L^1_{\loc}(\mathbb{M})$, $i=0,1$. 

We first need to establish some

\subsection{Preliminary estimates}

\subsubsection{A priori estimate} 

For $\psi\in \Psi$, let us introduce the terms
\begin{eqnarray}
\label{A1psi}A_1(\psi)&=& \int_{\overline{Q}}V^{\frac{-1}{p-1}}\psi^{\frac{-1}{p-1}}\left|\psi_{tt}\right|^{\frac{p}{p-1}}\,dx\,dt,\\
\label{A2psi}A_2(\psi)&=& \int_{\overline{Q}}V^{\frac{-1}{p-1}}\psi^{\frac{-1}{p-1}}\left|\psi_t\right|^{\frac{p}{p-1}}\,dx\,dt,\\
\label{A3psi}A_3(\psi)&=& \int_{\overline{Q}}V^{\frac{-1}{p-1}}\psi^{\frac{-1}{p-1}}\left|\Delta \psi\right|^{\frac{p}{p-1}}\,dx\,dt,\\
\label{A4psi}A_4(\psi)&=& \int_{\overline{Q}}V^{\frac{-1}{p-1}}\psi^{\frac{-1}{p-1}}\left|\Delta \psi_t\right|^{\frac{p}{p-1}}\,dx\,dt.
\end{eqnarray}
We have the following a priori estimate.

\begin{lemma}\label{PR4.1}{\rm{
If there exists	$u\in L^p_{\loc}\left(Q,V\,dx\,dt\right)\cap L^1_{\loc}(Q)$ which is a weak solution to \eqref{P1}, then 
\begin{eqnarray}\label{apest-P1}
\frac{1}{2}\int_{Q}|u|^p \psi V\,dx\,dt\nonumber &&+\int_{\mathbb{M}}u_0(x)\left(\psi(0,x)-\psi_t(0,x)-\Delta \psi(0,x)\right)\,dx\\ \nonumber &&+\int_{\mathbb{M}}u_1(x)\psi(0,x)\,dx\\
\nonumber &&+\int_Q W(t,x)\psi(t,x)\,dx\,dt\\
&&\leq C\sum_{j=1}^4A_j(\psi),	
\end{eqnarray}
for every $\psi\in \Psi$, provided that $A_j(\psi)<\infty$ for all $j=1,2,3,4$. }}
\end{lemma}

\begin{proof}
Let 	$u\in L^p_{\loc}\left(Q,V\,dx\,dt\right)\cap L^1_{\loc}(Q)$ be a weak solution to \eqref{P1} and $\psi\in \Psi$ with $A_j(\psi)<\infty$ for all $j=1,2,3,4$.  By \eqref{ws1}, we have
\begin{equation}\label{S1-PR4.1}
\begin{aligned}
\int_{Q}|u|^p \psi V\,dx\,dt &+\int_{\mathbb{M}}u_0(x)\left(\psi(0,x)-\psi_t(0,x)-\Delta \psi(0,x)\right)\,dx\\
&+\int_{\mathbb{M}}u_1(x)\psi(0,x)\,dx+\int_Q W\psi\,dx\,dt\\
&	\leq \int_{Q}|u|\left|\psi_{tt}\right|\,dx\,dt+ \int_{Q}|u|\left|\psi_{t}\right|\,dx\,dt\\&+\int_{Q}|u|\left|\Delta \psi\right|\,dx\,dt+\int_{Q}|u|\left|\Delta \psi_{t}\right|\,dx\,dt.
\end{aligned}
\end{equation}	
Making use of Young's inequality,  we obtain
\begin{equation}\label{S2-PR4.1}
\begin{aligned}
\int_{Q}|u|\left|\psi_{tt}\right|\,dx\,dt&= \int_{Q}|u|(\psi V)^{\frac{1}{p}}\left|\psi_{tt}\right|(\psi V)^{\frac{-1}{p}}\,dx\,dt\\
&\leq \frac{1}{8}\int_{Q}|u|^p \psi V\,dx\,dt+C A_1(\psi).
\end{aligned}
\end{equation}
Similarly, we have
\begin{eqnarray}
\label{S3-PR4.1} \int_{Q}|u|\left|\psi_{t}\right|\,dx\,dt&\leq & \frac{1}{8}\int_{Q}|u|^p \psi V\,dx\,dt+C A_2(\psi),\\
\label{S4-PR4.1}\int_{Q}|u|\left|\Delta \psi\right|\,dx\,dt &\leq & \frac{1}{8}\int_{Q}|u|^p \psi V\,dx\,dt+C A_3(\psi),\\	
\label{S5-PR4.1}\int_{Q}|u|\left|\Delta \psi_t\right|\,dx\,dt &\leq & \frac{1}{8}\int_{Q}|u|^p \psi V\,dx\,dt+C A_4(\psi).
\end{eqnarray}
Finally, \eqref{apest-P1} follows from \eqref{S1-PR4.1}, \eqref{S2-PR4.1}, \eqref{S3-PR4.1}, \eqref{S4-PR4.1} and \eqref{S5-PR4.1}.
\end{proof}

\subsubsection{Test function}
Let us introduce a cut-off function $\xi\in C^\infty([0,\infty))$ satisfying the following properties:
$$
0\leq \xi\leq 1,\,\, \xi\equiv 1 \mbox{ in } \left[0,\frac{1}{2}\right],\,\, \xi\equiv 0 \mbox{ in } [1,\infty).	
$$
For $T>0$ and  $\iota\gg 1$,   let 
$$
\xi_T(t)=\xi^\iota\left(T^{-1}t\right),\quad t\geq 0.
$$
For $R> 1$, let 
\begin{equation}\label{phiR}
\varphi_R(x)=\psi_R^\iota(x),\quad x\in \mathbb{M},
\end{equation}
where $\{\psi_R\}_{R>1}$ is the family of cut-off functions provided by Lemma \ref{L2.1}.  We consider test functions of the form
\begin{equation}\label{testf-P1}
\psi(t,x)=\xi_T(t)\varphi_R(x),\quad (t,x)\in Q.	
\end{equation}
By the properties of the function $\xi$ and Lemma \ref{L2.1}, we obtain the following result.

\begin{lemma}\label{LL4.2}{\rm{
For all $T>0$, $R> 1$ and $\iota\gg1$, the function $\psi$ defined by \eqref{testf-P1} belongs to $\Psi$. 
}}
\end{lemma}

\subsubsection{Estimates of $A_j(\psi)$} 

For $T,R,\iota\gg 1$, let $\psi$ be the test function defined by \eqref{testf-P1}. In this part we shall estimate the terms $A_j(\psi)$, $j=1,2,3,4$.  Let $\gamma_0>1$ be the parameter provided by Lemma \ref{L2.1}.

\begin{lemma}\label{LL4.3}{\rm{
Let $V^{\frac{-1}{p-1}}\in L^1_{\loc}(\mathbb{M})$. For  $\gamma>\gamma_0$, we have 
\begin{equation}\label{est-LL4.3}
A_1(\psi)\leq C T^{1-\frac{2p}{p-1}}	\int_{B(x_0,\gamma R)} V^{\frac{-1}{p-1}}(x)\,dx.
\end{equation}
}}
\end{lemma}

\begin{proof}
By \eqref{A1psi} and \eqref{testf-P1}, we have
\begin{equation}\label{S1-L44.3}
A_1(\psi)=\left(\int_0^\infty 	\xi_T^{\frac{-1}{p-1}}(t) \left|\xi_T''(t)\right|^{\frac{p}{p-1}}\,dt\right)\left(\int_{\mathbb{M}} V^{\frac{-1}{p-1}}(x) \varphi_R(x)\,dx\right). 
\end{equation}
On the other hand, by the definition of the function $\xi_T$ and the properties of the cut-off function $\xi$, we have
\begin{equation}\label{S2-L44.3}
\begin{aligned}
\int_0^\infty 	\xi_T^{\frac{-1}{p-1}}(t) \left|\xi_T''(t)\right|^{\frac{p}{p-1}}\,dt &\leq C 
\int_{\frac{T}{2}}^T  \xi^{\frac{-\iota}{p-1}}(T^{-1}t)\left|T^{-2}\xi^{\iota-2}(T^{-1}t)\right|^{\frac{p}{p-1}}\,dt\\
&=C T^{\frac{-2p}{p-1}}\int_{\frac{T}{2}}^T\xi^{\iota-\frac{2p}{p-1}}(T^{-1}t)\,dt\\
&\leq C T^{1-\frac{2p}{p-1}}.
\end{aligned}
\end{equation}
Furthermore, by the definition of the function $\varphi_R$ and Lemma \ref{L2.1}, for $\gamma>\gamma_0$, we have
\begin{equation}\label{S3-L44.3}
\begin{aligned}
\int_{\mathbb{M}} V^{\frac{-1}{p-1}}(x) \varphi_R(x)\,dx&=\int_{B(x_0,\gamma R)} V^{\frac{-1}{p-1}}(x) \psi_R^\iota(x)\,dx	\\
&\leq \int_{B(x_0,\gamma R)} V^{\frac{-1}{p-1}}(x)\,dx.
\end{aligned}
\end{equation}
Therefore, \eqref{est-LL4.3} follows from \eqref{S1-L44.3}, \eqref{S2-L44.3} and \eqref{S3-L44.3}. 
\end{proof}

Similarly, by \eqref{A2psi} and \eqref{testf-P1}, we obtain the following estimates.

\begin{lemma}\label{LL4.4}{\rm{
Let $V^{\frac{-1}{p-1}}\in L^1_{\loc}(\mathbb{M})$. For  $\gamma>\gamma_0$, we have 
$$
A_2(\psi)\leq C T^{1-\frac{p}{p-1}}	\int_{B(x_0,\gamma R)} V^{\frac{-1}{p-1}}(x)\,dx.
$$
}}
\end{lemma}

\begin{lemma}\label{LL4.5}{\rm{
Let $V^{\frac{-1}{p-1}}\in L^1_{\loc}(\mathbb{M})$. For   $\gamma>\gamma_0$,  we have 
\begin{equation}\label{est-LL4.5}
A_3(\psi)\leq C TR^{-\frac{\left(1+\frac{\min\{\sigma,2\}}{2}\right)p}{p-1}}	\int_{B(x_0,\gamma R)} V^{\frac{-1}{p-1}}(x)\,dx.
\end{equation}
}}
\end{lemma}

\begin{proof}
By \eqref{A3psi} and \eqref{testf-P1}, we have
\begin{equation}\label{S1-LL45}
A_3(\psi)=\left(\int_0^\infty \xi_T(t)\,dt\right)\left(\int_{\mathbb{M}} V^{\frac{-1}{p-1}}(x)\varphi_R^{\frac{-1}{p-1}}(x)|\Delta \varphi_R(x)|^{\frac{p}{p-1}}\,dx\right).	
\end{equation}
On the other hand, by the definition of the function $\xi_T$ and the properties of the cut-off function $\xi$, we have
\begin{equation}\label{S11-LL45}
\begin{aligned}
\int_0^\infty \xi_T(t)\,dt&= \int_0^T \xi^\iota(T^{-1}t)\,dt\\
&\leq CT.	
\end{aligned}
\end{equation}
Furthermore, by the definition of $\varphi_R$, for all $x\in \mathbb{M}$, we have
$$
\begin{aligned}
\Delta \varphi_R(x)&=\Delta\left(\psi_R^\iota(x)\right)\\
&=\iota \psi_R^{\iota-1}(x)\Delta \psi_R(x)+\iota(\iota-1)\psi_R^{\iota-2}(x)|\nabla\psi_R(x)|^2,
\end{aligned}	
$$
which implies by Lemma \ref{L2.1} and Remark \ref{RK2.2} that for $\gamma>\gamma_0$, 
\begin{equation}\label{S2-LL45}
\int_{\mathbb{M}} V^{\frac{-1}{p-1}}(x)\varphi_R^{\frac{-1}{p-1}}(x)|\Delta \varphi_R(x)|^{\frac{p}{p-1}}\,dx=\int_{B(x_0,\gamma R)\backslash B(x_0,R)} V^{\frac{-1}{p-1}}(x)\varphi_R^{\frac{-1}{p-1}}(x)|\Delta \varphi_R(x)|^{\frac{p}{p-1}}\,dx	
\end{equation}
and
\begin{equation}\label{oui}
\begin{aligned}
|\Delta \varphi_R(x)| &\leq C  \psi_R^{\iota-2}(x)\left(|\Delta\psi_R(x)|+|\nabla\psi_R(x)|^2\right)\\
&\leq C  \psi_R^{\iota-2}(x)\left(R^{-\left(1+\frac{\min\{\sigma,2\}}{2}\right)}+R^{-2}\right)\\
&\leq C R^{-\left(1+\frac{\min\{\sigma,2\}}{2}\right)} \psi_R^{\iota-2}(x)
\end{aligned}
\end{equation}
for all $x\in B(x_0,\gamma R)\backslash B(x_0,R)$. The above estimate yields
\begin{equation}\label{S3-LL45}
\begin{aligned}
\varphi_R^{\frac{-1}{p-1}}(x)|\Delta \varphi_R(x)|^{\frac{p}{p-1}}& \leq C R^{-\frac{\left(1+\frac{\min\{\sigma,2\}}{2}\right)p}{p-1}}\psi_R^{\iota-\frac{2p}{p-1}}(x)\\
&\leq C R^{-\frac{\left(1+\frac{\min\{\sigma,2\}}{2}\right)p}{p-1}}
\end{aligned}
\end{equation}
for all $x\in B(x_0,\gamma R)\backslash B(x_0,R)$.  Finally, \eqref{est-LL4.5} follows from \eqref{S1-LL45}, \eqref{S11-LL45}, \eqref{S2-LL45} and \eqref{S3-LL45}. The proof is complete.
\end{proof}

\begin{lemma}\label{LL4.6}{\rm{
Let $V^{\frac{-1}{p-1}}\in L^1_{\loc}(\mathbb{M})$. For  $\gamma>\gamma_0$,  we have 
\begin{equation}\label{est-LL4.6}
A_4(\psi)\leq C T^{1-\frac{p}{p-1}}R^{-\frac{\left(1+\frac{\min\{\sigma,2\}}{2}\right)p}{p-1}}	\int_{B(x_0,\gamma R)} V^{\frac{-1}{p-1}}(x)\,dx.
\end{equation}
}}
\end{lemma}

\begin{proof}
By \eqref{A4psi} and \eqref{testf-P1}, we have
\begin{equation}\label{S1-LL46}
A_4(\psi)=\left(\int_0^\infty \xi_T^{\frac{-1}{p-1}}(t) \left|\xi_T'(t)\right|^{\frac{p}{p-1}}\,dt\right)\left(\int_{\mathbb{M}} V^{\frac{-1}{p-1}}(x)\varphi_R^{\frac{-1}{p-1}}(x)|\Delta \varphi_R(x)|^{\frac{p}{p-1}}\,dx\right).	
\end{equation}
On the other hand, by the definition of the function $\xi_T$ and the properties of the cut-off function $\xi$, we have
\begin{equation}\label{S2-LL46}
\int_0^\infty \xi_T^{\frac{-1}{p-1}}(t) \left|\xi_T'(t)\right|^{\frac{p}{p-1}}\,dt\leq CT^{1-\frac{p}{p-1}}.	
\end{equation}
Finally, \eqref{est-LL4.6} follows from \eqref{S2-LL45}, \eqref{S3-LL45}, \eqref{S1-LL46} and \eqref{S2-LL46}. The proof is complete.
\end{proof}

\subsection{The homogeneous case ($W\equiv 0$)}
We now give the 
\subsubsection*{Proof of Theorem \ref{T3.2}}
We use the contradiction argument. Let us suppose that $u\in L^p_{\loc}\left(Q,V\,dx\,dt\right)\cap L^1_{\loc}\left(Q\right)$ is a weak solution to \eqref{P1}. By Lemma \ref{PR4.1} (with $W\equiv 0$) and Lemma \ref{LL4.2}, for $T,R,\ell\gg 1$, we obtain 
\begin{equation}\label{S1-TT1}
\int_{\mathbb{M}}u_0(x)\left(\psi(0,x)-\psi_t(0,x)-\Delta \psi(0,x)\right)\,dx+\int_{\mathbb{M}}u_1(x)\psi(0,x)\,dx\leq C\sum_{j=1}^4A_j(\psi),
\end{equation}
where $\psi$ is the test function given by \eqref{testf-P1}. On the other hand, by the definition of the function $\psi$ and the properties of the function $\xi$, we have
\begin{equation}\label{S2-TT1}
\psi_t(0,x)=\xi_T'(0)\varphi_R(x)=0,\quad x\in \mathbb{M}. 	
\end{equation}
Furthermore, by Lemma \ref{L2.1}, for $\gamma>\gamma_0$,  we have
\begin{equation}\label{S3-TT1}
\begin{aligned}
\int_{\mathbb{M}}u_0(x)\psi(0,x)\,dx&=\xi_T(0)\int_{B(x_0,\gamma R)}u_0(x)\varphi_R(x)\,dx\\
&=\xi^\iota(0) \int_{B(x_0,\gamma R)}u_0(x) \psi_R^\iota(x)\,dx\\
&=\int_{B(x_0,\gamma R)}u_0(x) \psi_R^\iota(x)\,dx. 
\end{aligned}
\end{equation}
Similarly, by Lemma \ref{L2.1}, for  $\gamma>\gamma_0$,  we have
\begin{equation}\label{S4-TT1}
\int_{\mathbb{M}}u_0(x)\Delta\psi(0,x)\,dx=\int_{B(x_0,\gamma R)\backslash B(x_0,R)}u_0(x)\Delta  (\psi_R^\iota)(x)\,dx
\end{equation}
and 
\begin{equation}\label{S5-TT1}
\int_{\mathbb{M}}u_1(x)\psi(0,x)\,dx=\int_{B(x_0,\gamma R)}u_1(x) \psi_R^\iota(x)\,dx. 	
\end{equation}
Next, in view of \eqref{S1-TT1}, \eqref{S2-TT1}, \eqref{S3-TT1}, \eqref{S4-TT1}, \eqref{S5-TT1} and Lemmas \ref{LL4.3}, \ref{LL4.4}, \ref{LL4.5}, \ref{LL4.6},   we get
$$
\begin{aligned}	
&\sum_{i=0}^1\int_{B(x_0,\gamma R)} u_i(x)\psi_R^\iota(x)\,dx-\int_{B(x_0,\gamma R)\backslash B(x_0,R)}u_0(x)\Delta  (\psi_R^\iota)(x)\,dx\\
&\leq  C \left(T^{1-\frac{p}{p-1}}	\int_{B(x_0,\gamma R)} V^{\frac{-1}{p-1}}(x)\,dx
+ TR^{-\frac{\left(1+\frac{\min\{\sigma,2\}}{2}\right)p}{p-1}}	\int_{B(x_0,\gamma R)} V^{\frac{-1}{p-1}}(x)\,dx\right).
 \end{aligned}
$$
Taking $T=R^{1+\frac{\min\{\sigma,2\}}{2}}$, the above estimate reduces to 
\begin{equation}\label{S6-TT1}
\begin{aligned}
&\sum_{i=0}^1\int_{B(x_0,\gamma R)} u_i(x)\psi_R^\iota(x)\,dx-\int_{B(x_0,\gamma R)\backslash B(x_0,R)}u_0(x)\Delta  (\psi_R^\iota)(x)\,dx\\
&\leq C R^{-\frac{\left(1+\frac{\min\{\sigma,2\}}{2}\right)}{p-1}}\int_{B(x_0,\gamma R)} V^{\frac{-1}{p-1}}(x)\,dx.
\end{aligned}
\end{equation}
On the other hand, by \eqref{oui}, and since $0\leq \psi_R\leq 1$ (by Lemma \ref{L2.1}), we have
$$
\begin{aligned}
\left|\int_{B(x_0,\gamma R)\backslash B(x_0,R)}u_0(x)\Delta  (\psi_R^\iota)(x)\,dx\right|&\leq \int_{B(x_0,\gamma R)\backslash B(x_0,R)}|u_0(x)|\left|\Delta  (\psi_R^\iota)(x)\right|\,dx	\\
&\leq C R^{-\left(1+\frac{\min\{\sigma,2\}}{2}\right)}\int_{B(x_0,\gamma R)\backslash B(x_0,R)}|u_0(x)|\,dx.
\end{aligned}
$$
Furthermore, since $u_0\in L^1(\mathbb{M})$, by the dominated convergence theorem,  we have
$$
\lim_{R\to \infty} R^{-\left(1+\frac{\min\{\sigma,2\}}{2}\right)}\int_{B(x_0,\gamma R)\backslash B(x_0,R)}|u_0(x)|\,dx=0,
$$
which yields
\begin{equation}\label{S7-TT1}
\lim_{R\to \infty} \int_{B(x_0,\gamma R)\backslash B(x_0,R)}u_0(x)\Delta  (\psi_R^\iota)(x)\,dx=0. 
\end{equation}
Using again the dominated convergence theorem and taking into consideration that $u_i\in L^1(\mathbb{M})$, $i=0,1$, it follows from Lemma \ref{L2.1} that 
\begin{equation}\label{S8-TT1}
\lim_{R\to \infty} \sum_{i=0}^1\int_{B(x_0,\gamma R)} u_i(x)\psi_R^\iota(x)\,dx=	\sum_{i=0}^1\int_{\mathbb{M}} u_i(x)\,dx. 
\end{equation}
Thus, taking the infimum limit as $R\to\infty$ in \eqref{S6-TT1}, in view of \eqref{cd-blow-up-P1},  \eqref{S7-TT1} and \eqref{S8-TT1}, we get
$$
\sum_{i=0}^1\int_{\mathbb{M}} u_i(x)\,dx\leq 0,
$$
which contradicts \eqref{cdu01}.  This completes the proof of Theorem \ref{T3.2}. \hfill $\square$

\subsection{The inhomogeneous case ($W\not\equiv 0$)}

We provide below the proof of Theorem \ref{T3.9}.

\subsubsection*{Proof of Theorem \ref{T3.9}}
We also use the contradiction argument. Let us suppose that $u\in L^p_{\loc}\left(Q,V\,dx\,dt\right)\cap L^1_{\loc}\left(Q\right)$ is a weak solution to \eqref{P1}. By Lemmas \ref{PR4.1} and \ref{LL4.2}, for $T,R,\ell\gg 1$, we obtain 
$$
\begin{aligned}
\int_{\mathbb{M}}u_0(x)\left(\psi(0,x)-\psi_t(0,x)-\Delta \psi(0,x)\right)\,dx&+\int_{\mathbb{M}}u_1(x)\psi(0,x)\,dx\\&+\int_Q W(t,x)\psi(t,x)\,dx\,dt\\
&\leq C\sum_{j=1}^4A_j(\psi)	,
\end{aligned}
$$
where $\psi$ is the test function given by \eqref{testf-P1}. Taking into consideration \eqref{S2-TT1}, \eqref{S4-TT1} and using \eqref{u01-cdP2}, we get
\begin{equation}\label{LALA1}
-\int_{B(x_0,\gamma R)\backslash B(x_0,R)}u_0(x)\Delta  (\psi_R^\iota)(x)\,dx+\int_Q W(t,x)\psi(t,x)\,dx\,dt\leq C\sum_{j=1}^4A_j(\psi)	
\end{equation}
for all $\gamma>\gamma_0$. Furthermore, since $W\geq 0$ a.e., making use of Lemma \ref{L2.1}, \eqref{testf-P1} and the properties of the cut-off function $\xi$, we obtain 
\begin{equation}\label{LALA2}
\begin{aligned}
\int_Q W(t,x)\psi(t,x)\,dx\,dt&\geq \int_0^{\frac{T}{2}}\int_{B(x_0,R)}	W(t,x)\,dx\,dt.
\end{aligned}
\end{equation}
Hence, it follows from \eqref{LALA1}, \eqref{LALA2} and Lemmas \ref{LL4.3}, \ref{LL4.4}, \ref{LL4.5}, \ref{LL4.6},   that 
$$
\begin{aligned}
& -\int_{B(x_0,\gamma R)\backslash B(x_0,R)}u_0(x)\Delta  (\psi_R^\iota)(x)\,dx+	\int_0^{\frac{T}{2}}\int_{B(x_0,R)}	W(t,x)\,dx\,dt\\
&\leq C \left(T^{1-\frac{p}{p-1}}	\int_{B(x_0,\gamma R)} V^{\frac{-1}{p-1}}(x)\,dx
+ TR^{-\frac{\left(1+\frac{\min\{\sigma,2\}}{2}\right)p}{p-1}}	\int_{B(x_0,\gamma R)} V^{\frac{-1}{p-1}}(x)\,dx\right),
\end{aligned}
$$
that is,
$$
\begin{aligned}
& \left(-\int_{B(x_0,\gamma R)\backslash B(x_0,R)}u_0(x)\Delta  (\psi_R^\iota)(x)\,dx\right)\left(	\int_0^{\frac{T}{2}}\int_{B(x_0,R)}	W(t,x)\,dx\,dt\right)^{-1}+1\\
&\leq C \left(T^{1-\frac{p}{p-1}}	\int_{B(x_0,\gamma R)} V^{\frac{-1}{p-1}}(x)\,dx
+ TR^{-\frac{\left(1+\frac{\min\{\sigma,2\}}{2}\right)p}{p-1}}	\int_{B(x_0,\gamma R)} V^{\frac{-1}{p-1}}(x)\,dx\right)\\
&\quad \cdot \left(	\int_0^{\frac{T}{2}}\int_{B(x_0,R)}	W(t,x)\,dx\,dt\right)^{-1}.
\end{aligned}
$$
Taking $T=R^{1+\frac{\min\{\sigma,2\}}{2}}$, the above estimate reduces to 
\begin{equation}\label{final-estimate}
\begin{aligned}
& \left(-\int_{B(x_0,\gamma R)\backslash B(x_0,R)}u_0(x)\Delta  (\psi_R^\iota)(x)\,dx\right)\left(	\int_0^{\frac{R^{1+\frac{\min\{\sigma,2\}}{2}}}{2}}\int_{B(x_0,R)}	W(t,x)\,dx\,dt\right)^{-1}+1\\
&\leq C R^{-\frac{\left(1+\frac{\min\{\sigma,2\}}{2}\right)}{p-1}}\left(\int_{B(x_0,\gamma R)} V^{\frac{-1}{p-1}}(x)\,dx\right) \left(	\int_0^{\frac{R^{1+\frac{\min\{\sigma,2\}}{2}}}{2}}\int_{B(x_0,R)}	W(t,x)\,dx\,dt\right)^{-1}.
\end{aligned}
\end{equation}
On the other hand, since $u_0\in L^1(\mathbb{M})$, \eqref{S7-TT1} holds. Furthermore, since $W\in L^1_{\loc}(Q)$, $W\geq 0$ a.e. and $W\not\equiv 0$,  then there exist $R_0,R_1>0$ such that 
$$
\int_0^{\frac{R^{1+\frac{\min\{\sigma,2\}}{2}}}{2}}\int_{B(x_0,R)}	W(t,x)\,dx\,dt\geq 	\int_0^{R_0}\int_{B(x_0,R_1)}	W(t,x)\,dx\,dt>0,
$$
which yields
$$
\begin{aligned}
&\left|\left(-\int_{B(x_0,\gamma R)\backslash B(x_0,R)}u_0(x)\Delta  (\psi_R^\iota)(x)\,dx\right)\left(	\int_0^{\frac{R^{1+\frac{\min\{\sigma,2\}}{2}}}{2}}\int_{B(x_0,R)}	W(t,x)\,dx\,dt\right)^{-1}\right|\\
&\leq C \left|\int_{B(x_0,\gamma R)\backslash B(x_0,R)}u_0(x)\Delta  (\psi_R^\iota)(x)\,dx\right|.
\end{aligned}
$$
Then, by \eqref{S7-TT1}, we deduce that 
\begin{equation}\label{limu0delta}
\lim_{R\to \infty} \left(-\int_{B(x_0,\gamma R)\backslash B(x_0,R)}u_0(x)\Delta  (\psi_R^\iota)(x)\,dx\right)\left(	\int_0^{\frac{R^{1+\frac{\min\{\sigma,2\}}{2}}}{2}}\int_{B(x_0,R)}	W(t,x)\,dx\,dt\right)^{-1}=0.
\end{equation}
We now take $\gamma>\gamma_0$ so that \eqref{cd2-blowup-P2} is satisfied. Then, passing to the infimum limit as $R\to \infty$ in \eqref{final-estimate}, using  \eqref{cd2-blowup-P2} and \eqref{limu0delta}, we reach a contradiction. This completes the proof of Theorem \ref{T3.9}. \hfill $\square$

\section*{Declaration of competing interest}
	The authors declare that there is no conflict of interest.

\section*{Data Availability Statements} The manuscript has no associated data.

\section*{Acknowledgments}
MR and BT was supported by the FWO Odysseus 1 grant G.0H94.18N: Analysis and Partial Differential Equations and by the Methusalem programme of the Ghent University Special Research Fund (BOF) (Grant number 01M01021). MR is also supported by EPSRC grant EP/R003025/2 and EP/V005529/1. BS is supported by Researchers Supporting Project number (RSP2023R4), King Saud University, Riyadh, Saudi Arabia. BT is also supported by the Science Committee of the Ministry of Education and Science of the Republic of Kazakhstan (Grant No. BR20281002).

\end{document}